\documentclass[12pt]{amsart}
\usepackage{geometry}   
\usepackage[colorlinks,citecolor = red, linkcolor=blue,hyperindex]{hyperref}
\usepackage{euscript,eufrak,verbatim, mathrsfs}
\usepackage[psamsfonts]{amssymb}
\usepackage[usenames]{color}
\usepackage{bbm}
\usepackage{graphicx}
 \usepackage{float}

 \usepackage[all, cmtip]{xy}

\usepackage{upref, xcolor, dsfont}
\usepackage{amsfonts,amsmath,amstext,amsbsy, amsopn,amsthm, amscd}
\usepackage{enumerate}

\usepackage{url}

\usepackage{mathtools}
\usepackage{bookmark}

 \usepackage{euscript}
\usepackage{helvet}         
\usepackage{courier}        
\usepackage{type1cm}        
\usepackage{multicol}        
\usepackage[bottom]{footmisc}

\newtheorem{theorem}{Theorem}[section]
\newtheorem*{theorem*}{Theorem B} 
\newtheorem{lemma}[theorem]{Lemma}

\newtheorem{proposition}[theorem]{Proposition}

\newtheorem*{definition*}{Definition}
\newtheorem*{remark*}{Remark}

\newtheorem*{observation*}{Observation}

\newtheorem*{assumption*}{Assumption}
\newtheorem*{question*}{Question}

\newcommand{\R}{\mathbb{R}}
\newcommand{\N}{\mathbb{N}}
\newcommand{\Z}{\mathbb{Z}}

\newcommand{\C}{\mathbb{C}}
\newcommand{\E}{\mathbb{E}}
\newcommand{\T}{\mathbb{T}}
\newcommand{\PP}{\mathbb{P}}

\newcommand{\X}{\mathfrak{X}}

\newcommand{\Cov}{\mathrm{Cov}}

\newcommand{\an}{\text{\, and \,}}

\newcommand{\m}{\mathfrak{m}}

\begin{document}

\title{Some  properties of stationary determinantal point processes on $\mathbb{Z}$}

\author
{Aihua Fan}
\address
{Aihua FAN:   LAMFA, UMR 7352, CNRS, University of Picardie, 33 Rue Saint Leu, Amiens, France \&
School of Mathematics and Statistics, Central China Normal University, 430079 Wuhan, China
}
\email{ai-hua.fan@u-picardie.fr}

\author
{Shilei Fan}
\address
{Shilei FAN: School of Mathematics and Statistics, \& Hubei Key Laboratory of Mathematical Sciences, Central China Normal University, Wuhan 430079, China}
\email{slfan@mail.ccnu.edu.cn}

\author
{Yanqi Qiu}
\address
{Yanqi QIU: Institute of Mathematics, Academy of Mathematics and Systems Science, Chinese Academy of Sciences, Beijing 100190, China}
\email{yanqi.qiu@hotmail.com}

\begin{abstract}
	We study  properties of stationary determinantal point processes $\X$ on $\Z$ from different points of views. It is proved that $\X\cap \N$
	is almost surely Bohr-dense  and good universal for almost everywhere convergence in $L^1$, and that $\X$ is not syndetic but $\X +\X = \mathbb{Z}$.  For the associated centered random field, we obtain
	a sub-Gaussian property, a Salem-Littlewood inequality and a Khintchine-Kahane inequality. Results can be generalized to $\Z^d$.  
\end{abstract}


\maketitle

\setcounter{equation}{0}

\section{Introduction}
Let $\T: = \R/\Z$ be the unit circle, equipped with the normalized  Lebesgue measure $d\m$ (we also use $dt$ for simplifying the notation $d\m(t)$).   For any non-negative Borel function $f: \T  \rightarrow [0, 1]$ such that 
\begin{align}\label{def-sigma}
\sigma: = \int_{\T} f d\m  \in (0, 1),
\end{align}
the kernel $K_f : \Z\times \Z \to \C$   defined by
\begin{align}\label{def-kernel}
K_f (n, m) : = \widehat{f}(n-m) = \int_{\T}  f(t) e^{- 2 \pi i  (n-m) t  } d \m(t) \quad (\forall n, m \in \Z)
\end{align}
determines a self-adjoint bounded operator on the Hilbert space $\ell^2(\Z)$
with spectrum contained in the interval $[0,1]$ and thus induces
a non-trivial stationary determinantal point process $\X$ on $\Z$, see Lyons and Steif \cite{Lyons-stationary} (we exclude the trival cases $\X=\emptyset $ or $\Z$ almost surely,  corresponding to $\sigma=0$ or $1$). More precisely, $\X = \X (\omega)$ is a random subset of $\Z$ (defined on a probability space $(\Omega, \mathcal{A}, \PP)$ where an elementary event in $\Omega$ is denoted by $\omega$), whose distribution is described as follows: if we identify $\X  = \X(\omega)$ with the family of random variables $\xi = (\xi_n(\omega))_{n\in \Z}$ taking values in $\{0, 1\}$ (i.e. a random field on ${\Z}$) in the following natural way: 
\begin{align}
\label{xi}
\xi_n(\omega) = 1   \text{\, iff \,}  n \in \X(\omega),
\end{align}
then for any distinct points  $n_1, \cdots, n_k \in \Z$, we have
\begin{align}\label{def-DPP-bis}
\E(\xi_{n_1} \cdots \xi_{n_k})=  \det (\widehat{f}(n_i - n_j))_{1\le i, j \le k}.
\end{align}
The distribution of $\xi = (\xi_n)_{n\in \Z}$ satisfying \eqref{def-DPP-bis} will be denoted by $\nu_f$, which is a probability Borel measure on $\{0, 1\}^{\Z}$.

When $f$ is equal to a constant $p\in (0,1)$, $\X$ is nothing but the $p$-Bernoulli set, corresponding to the i.i.d.
sequence $(\xi_n)_{n \in \Z}$ such that $\PP(\xi_0 =1) = p =1-\PP(\xi_0=0)$. When $f = \mathds{1}_{[-1/4, 1/4]}$,  the corresponding determinantal point process is the discrete Dyson-sine process induced by the discrete sine kernel: 
\[
K_{[-1/4, 1/4]}(x, y):= \frac{\sin (\frac{\pi}{2}(x -y))}{\pi ( x -y) },\ \ \  \forall (x, y) \in \Z\times \Z \setminus \{(0,0)\}, 
\]
and $K_{[-1/4, 1/4]}(0,0) =\frac{1}{2}$. See \cite{Johansson-sine} for some recent results on Dyson-sine process.
\medskip

In the present paper, we will study some properties of $\X$ as subset of $\Z$ from ergodic theory point of view, functional analysis point of view and arithmetic point of view. 

Let us explain briefly some of our results. Denote $\X_{+}: = \X \cap \N$,  the non-negative part of $\X$. 
 Our first result (Theorem \ref{exactness})
states that 
\begin{equation}\label{eq-exactness}
     a.s. \quad      \lim_{N\to\infty} \frac{1}{N} \sum_{n=0}^{N-1} \xi_n(\omega) e^{2\pi i n t} =0 \ \ \ \forall t \in (0,1).
\end{equation}
That is to say, $\X_+$ is almost surely uniformly distributed on the 
Bohr group, the dual group of $\T$ equipped with the discrete topology, in which $\Z$ is dense.    
In other words, $\X_+$ is a.s. $L^2$-exact in the sense 
of Fan and Schneider \cite[pp. 641-642]{Fan-Schneider}. That is to say, there exists
an event $\Omega_0 \subset \Omega$ with $\PP(\Omega_0)=1$ such that each $\omega \in \Omega_0$ has the following property:  for any measure-preserving dynamical system $(X, \mathcal{B}, \nu, T)$ and any $g\in L^2(\nu)$ we have
$$ 
  L^2-\lim_{N\to\infty} \frac{1}{N} \sum_{n=0}^{N-1} \xi_n(\omega) g(T^n x) = \mathbb{E}(g| \mathcal{J}_T)
$$
where $\mathcal{J}_T$ is the sub-$\sigma$-field of $T$-invariant sets.
We can even prove  that for any integer-coefficient polynomial $P \in \Z[x]$ with degree $d:=\deg P \ge 1$, $P(\X_+)$ is a.s. $L^2$-exact. That is to say
\begin{equation}\label{eq-exactness2}
a.s. \quad
\lim_{N\to\infty} \frac{1}{N} \sum_{n=0}^{N-1} \xi_n(\omega) e^{2\pi i P(n) t} =0  \ \ \ \forall t \in (0,1).
\end{equation}
This is a consequence of  the following inequality of Salem-Littlewood type (Theorem \ref{prop-quasi-G}):
\begin{align}\label{SL-DPP}
\max_{0 \le t \le 1} \left|  \sum_{n=0}^{N-1} (\xi_n(\omega) - \sigma) e^{2 \pi i P(n) t}  \right| \le C_\omega   \sqrt{d N   \log N}.
\end{align}
The classical Salem-Littlewood inequality concerns the $L^\infty$-norm estimate of random trigonometric polynomial with independent coeffcients, see Kahane \cite[Chapter 6]{Kahane-random}. But, in (\ref{SL-DPP}),
the coefficients $\xi_n -\sigma$ are not independent. However 
they do share the sub-Gaussian property with independent variables
(Proposition \ref{lem-qG}):
\begin{align}\label{inq-qG}
\E \left[ \exp\left(\lambda \sum_{n = 0}^{N-1}   (\xi_n(\omega) - \sigma) a_n \right ) \right] \le \exp \left(  \lambda^2 \sum_{n=0}^{N-1} a_n^2\right),
\end{align}
for all $\lambda\in \mathbb{R}$ and all $(a_0, \cdots, a_{N-1}) \in \mathbb{R}^N$. 

As a consequence of (\ref{SL-DPP}) and a recent result in Fan \cite[Theorem 2]{Fan-Osc}, we prove that for any integral polynomial $P\in \Z[x]$ with $\deg P\ge 1$ and $P(\N) \subset \N$, the set $P(\X_+)$ is a.s.  good universal for almost
everywhere convergence in $L^r$  with $r>1$ (we can take $r=1$ when $\deg P=1$).  
More precisely, there exists
an event $\Omega_0 \subset \Omega$ with $\PP(\Omega_0)=1$ such that each $\omega \in \Omega_0$ has the following property:  for any measure-preserving dynamical system $(X, \mathcal{B}, \nu, T)$ and any $g\in L^r(\nu)$  we have (Theorem \ref{thm-erg})
$$ 
\text{$\nu$-a.e. $x\in X$,} \ \ \ \ 
\lim_{N\to\infty} \frac{1}{N} \sum_{n=0}^{N-1} \xi_n(\omega) g(T^{P(n)} x) 
=\sigma   \lim_{N \to \infty}  \frac{1}{ N} \sum_{n = 0}^{N-1}  g(T^{P(n)} x).
$$ 
The almost everywhere convergence on the RHS is ensured by Bourgain's theorem \cite[Theorem 1 and Theorem 2]{Bourgain1989-IHES}. 
When $\deg P=1$, our Theorem \ref{thm-erg} is essentially a direct consequence of Bourgain's return time theorem, see \cite{BFKO1989}.

Assume $P\in \Z[x]$ 	with $\deg P \ge 1$ and $P(\N)\subset \N$. 
As an $L^2$-exact sequence, $\X_{+}$ is ergodic for finite periodic systems so that a.s.   for any integer $q$ and $a$ with $q\not=0$,  $P(\X_{+})$ has an infinite intersection with the arithmetic sequence $q \mathbb{N} +a$. Even we have quantitatively 
$$
\forall 0\le a <q, \ \ \ \lim_{N\to \infty} \frac{\sharp(\{n\in P(\X_{+}) \cap [1, N]: n=a \mod q\})}{ \sharp(\{ P(\X_{+}) \cap [1, N]\})} = \frac{1}{q}.
$$
See (1.3) in \cite[p. 15]{Rosenblatt-Wierdl}.
We have also seen that $P(\X_+)$ is even Bohr dense.
These show the richness of $\X$. However, $\X_{+}$ is not syndetic, namely there are gaps in $\X_+$ as large as possible (Theorem \ref{prop-syndetic}).
Recall that an increasing sequence of integers $(u_n) \subset \mathbb{N}$ is syndetic if it has bounded gaps, i.e. $\sup_{n} (u_{n+1}-u_n)<\infty$.

We  also prove  that $(\xi_n -\sigma)$ is a Riesz system in $L^2(\PP)$
iff $f$ is not an indicator function (Theorem \ref{thm-L2}). More precisely we prove
that 
\begin{align}\label{Riesz}
c_f \left( \sum_n | a_n|^2 \right)^{1/2}\le 
\left\|\sum_n a_n (\xi_n - \sigma) \right\|_2 \le   \sqrt{\sigma} \left( \sum_n | a_n|^2 \right)^{1/2}
\end{align}
for all complex sequence $(a_n) \in \ell^2(\Z)$,
where 
\begin{align}\label{c-f}
c_f = \sqrt{\int_{\T} f(t) (1-f(t)) d\m(t)}. 
\end{align}
Using Theorem 1.1 in Fan \cite{Fan2017ETDS}, together with (\ref{Riesz}), we deduce immediately that 
the random series $\sum a_n (\xi_n -\sigma)$ converges a.s. iff
$\sum |a_n|^2<\infty$.
Notice that $c_f=0$ iff $f$ is an indicator function. Also
notice that the RHS inequality in (\ref{Riesz}) always holds for any $f$. 

 Under the condition $c_f>0$ and $2\le p <\infty$, we also prove the following 
Khintchine-Kahane inequality (Theorem \ref{thm-KK}):
\begin{align}\label{KK-inquality}
\left\| \sum_{n} a_n  (\xi_n - \sigma) \right\|_2 \le \left\| \sum_{n} a_n  (\xi_n - \sigma) \right\|_p \le C_f(p)\left\| \sum_{n} a_n  (\xi_n - \sigma) \right\|_2,
\end{align} 
for all   $(a_n) \in \ell^2(\Z)$,  where 
$$
C_f(p) = \frac{\sqrt{2} e^{3/p} \Gamma(p+1)^{1/p}}{ \sqrt{\int_{\T} f(1-f) d\m}}.
$$

The above results show that although $\xi_n$'s are not independent, they do share many properties with independent 
variables.

\section{Recurrence, $L^2$-exactness and syndetic property }
\subsection{Recurrence and $L^2$-exactness}
Let $\Lambda = \{u_n\}_{n\ge 0}$ be a strictly increasing sequence of positive integers. $\Lambda$ is called a {\it Poincar\'e set} or $1$-{\it recurrent set} or simply {\it recurrent set} if for any measure-preserving system $(X, \mathcal{B}, \nu, T)$ and any $A \in \mathcal{B}$ with $\nu(A)>0$, there exists $n \in \Lambda$ such that 
\[
\nu(A \cap T^{-n} A) > 0.
\]
Following Fan-Schneider  \cite{Fan-Schneider}, $\Lambda$ is said to be $L^2$-exact if for any measure-preserving system $(X, \mathcal{B}, \nu, S)$ and any $g \in L^2(\nu)$, the following average
\[
A_N^\Lambda g (x): = \frac{1}{N} \sum_{n = 0}^{N-1}  g(T^{u_n} x) 
\] 
converges in $L^2(\nu)$ to $\E(g|\mathcal{J}_T)$, where $\E(\cdot| \mathcal{J}_T)$ denotes the conditional expectation with respect to the invariant $\sigma$-field $\mathcal{J}_T$ of $T$. 

Fan-Schneider showed in \cite[Theorem B]{Fan-Schneider} that any $L^2$-exact sequence is recurrent and  in \cite[Theorem 3.2]{Fan-Schneider} that  $\Lambda$ is $L^2$-exact if and only if the following average 
\begin{align}\label{wiener-av}
    \widehat{A}_N^\Lambda (t)  : = \frac{1}{\# (\Lambda \cap [0, N-1])}  \sum_{u \in \Lambda \cap [0, N-1] } e^{2 \pi i u t} 
\end{align}
converges to $0$ as $N \to \infty$ for all $t \in (0, 1)$. 
So, being $L^2$-exact is equivalent to being ergodic in the sense of 
\cite{Rosenblatt-Wierdl}. 


The following Theorem \ref{exactness} is a warm-up. A much stronger result, Theorem \ref{thm-erg},  will be proved. 

\begin{theorem}\label{exactness}
Let $\X$ be the determinantal point process induced by the kernel $K_f$ defined by \eqref{def-kernel}, associated to an arbitrary function $f: \T \to [0,1]$ such that $0<\sigma :=\int f d\m <1$.  Then almost surely,  $\X_{+}$ is $L^2$-exact and therefore is recurrent.  
\end{theorem}

\begin{proof}
Recall that we identify $\X$ with a random element $\xi =(\xi_n)_{n\in\Z}\in \{0, 1\}^\Z$ and the distribution of $\xi$ is denoted by $\nu_f$. Under this identification, we may write
\[
    \widehat{A}_N^{\X_{+}}(t)  =  \frac{1}{\# (\X \cap [0, N-1])}   \sum_{n=0}^{N-1} e^{2 \pi i  n t}  \xi_n.
\]

Consider the measure-preserving system $(\{0, 1\}^\Z, \nu_f, S)$ where 
 $S$ is the usual shift operator.  This dynamical system is strongly mixing (it is even conjugate to a Bernoulli shift, see Lyons-Steif \cite[Theorem 3.1]{Lyons-stationary}). Hence, by Birkhoff ergodic theorem, for $\nu_f$-almost all $\xi$ we have 
\begin{equation}\label{density}
\lim_{N\to\infty}\frac{\# (\X \cap [0, N-1])}{N}   = \lim_{N\to\infty}\frac{1}{N} \sum_{n=0}^{N-1} \pi_0(S^n \xi) = \sigma >0,
\end{equation}
where $\pi_0: \{0, 1\}^\Z \rightarrow \{0,1\}$ is the projection to the $0$-th coordinate. That means $\sigma$ is the density of $\X_+$.
Therefore, to show the $L^2$-exactness of $\X_{+}$, it suffices to show that
\begin{align}\label{N-sigma}
\text{$\nu_f$-a.e. $\xi \in \{0, 1\}^\Z$},  \quad \lim_{N\to\infty}  \frac{1}{N}   \sum_{n=0}^{N-1} e^{2 \pi i  n t}     \pi_0 (S^n \xi) =  0  \text{\, for all $t \in (0, 1)$.}
\end{align}

Set 
\[
g(\xi): = \xi_0  - \E(\xi_0).
\]
We have  only to prove
\begin{align}\label{W-W}
a.e. \ \ \ \lim_{N\to\infty}\sup_{t}  \left| \frac{1}{N} \sum_{n=0}^{N-1} e^{2 \pi i  n t}   g(S^n \xi)\right| = 0.
\end{align}

As we have mentioned above, the measure preserving system $(\{0, 1\}^\Z, \nu_f, S)$ is ergodic and even  strongly mixing \cite[Theorem 3.1]{Lyons-stationary}.  Then the Kronecker factor of the system 
consists only of constant functions.
Since   $\int g d\nu_f = 0$, namely $g$ is orthogonal to the Kronecker factor,
we may apply the Bourgain's uniform Wiener-Wintner ergodic theorem  \cite{Bourgain1990} to obtain the desired \eqref{W-W}.
To finish the proof, we  write 
\[
\frac{1}{N} \sum_{n=0}^{N-1} e^{2 \pi i  n t}  \xi_n =   \frac{1}{N} \sum_{n=0}^{N-1} e^{2 \pi i  n t}   g (S^n \xi) +  \frac{\sigma}{N} \sum_{n=0}^{N-1} e^{2 \pi i  n t}.
\]
It is clear that the desired  result \eqref{N-sigma} follows from  \eqref{W-W} and the simple fact
\[
\lim_{N\to \infty}  \frac{1}{N} \sum_{n=0}^{N-1} e^{2 \pi i  n t} =0 \text{\, for all $t \in (0, 1).$}
\]
\end{proof}


Not only $g$ is orthogonal to the Kronecker factor, but also
its spectral measure is absolutely continuous with respect to Lebesgue measure.  
Recall that the spectral measure $\sigma_{g}$ of $g$ is the unique measure on $\T$ such that its Fourier coefficients are given by 
\[
\widehat{\sigma}_{g} (n) = \int g \circ S^n \cdot \overline{g} d\nu \quad \forall n \in \Z. \]
In other words, we have $\widehat{\sigma}_{g} (n)  = \Cov(\xi_{n}, \xi_0)$.
Since the system $(\{0, 1\}^\Z, \nu_f, S)$ is mixing, $\int g d\nu_f=0$
implies that the spectral measure $\sigma_g$ is continuous, i.e. without atoms. We have more than that.

\begin{proposition} For $g = \xi_0 - \mathbb{E}\xi_0$, the spectral
	measure $\sigma_g$ is absulutely continuous and $\frac{d\sigma_g}{d\m}
	\in A(\T)$, where $A(\T)$ is the space of  continuous functions on $\T$ whose Fourier coefficients are summable.
	\end{proposition}

\begin{proof}
For $n \ne 0$, using the determinantal structure, we have 
\begin{align*}
\E(\xi_{n}\xi_0) =   \left| \begin{array}{cc}  \widehat{f}(0-0) &  \widehat{f} (0-n) \\  \widehat{f}(n-0) &  \widehat{f}(n-n) \end{array} \right|  =  \widehat{f}(0)^2 - | \widehat{f}(n)|^2.
\end{align*}
Combining this with the equality $\E(\xi_n) = \E(\xi_0) = \widehat{f}(0)$, we obtain 
\[
\widehat{\sigma}_{g} (n)   = \E(\xi_n \xi_0) - \E(\xi_n)\E(\xi_0)=- | \widehat{f}(n)|^2  \text{\, for all $n\ne 0$}.
\]
Consequently, since $f \in L^\infty(\T) \subset L^2(\T)$, we have
\[
\sum_{n\in \Z} |\widehat{ \sigma_g}(n)|  =  | \widehat{\sigma_g}(0)| + 2 \sum_{n=1}^\infty | \widehat{f}(n)|^2 < \infty.
\]
This  implies that $\sigma_g$ is absolutely continuous with respect to $\m$ and
\[
\frac{d\sigma_g}{d\m}   (e^{2 \pi i t}) = \sum_{n\in \Z} \widehat{\sigma_g}(n) e^{2 \pi i n t} \in A(\T).
\]
\end{proof}

\subsection{$\X$ is not syndetic}

\begin{theorem}\label{prop-syndetic}
Almost surely,  $\X_{+}$ is not syndetic.  
\end{theorem}
For the proof of Theorem \ref{prop-syndetic}, we will need the following elementary fact \eqref{pp-gap} about the gap probability of $\X_+$. Note that the positivity \eqref{pp-gap} follows immediately from Lyons \cite[Theorem 4.2]{Lyons-stationary}. To be self-containing, we present here an alternative proof of \eqref{pp-gap}.

\begin{lemma}\label{lem-contractive}
For any positive integer $\ell \ge 1$,  the positive contractive matrix 
$$
 \mathds{1}_{\{0, \cdots, \ell-1 \}} K_f \mathds{1}_{\{0, \cdots, \ell-1\}} = [\widehat{f}(n-m)]_{0\le n, m \le \ell -1}
$$
is strictly contractive. Consequently, we have  
\begin{align}\label{pp-gap}
\mathbb{P}(\X \cap \{0, 1, \cdots, \ell-1 \} =\emptyset)  = \det (1 -  \mathds{1}_{\{0, \cdots, \ell-1 \}} K_f \mathds{1}_{\{0, \cdots, \ell-1\}} ) > 0. 
\end{align}
\end{lemma}

\begin{proof} Consider the finite-dimensional subspace of trigonometric polynomials 
$$
  \mathscr{P}_\ell: = \{ P \in L^2(\T) : P (t) = \sum_{n=0}^{\ell-1} a_n e^{2 \pi i n t}, a_n \in \C \}
$$
and let $\Pi_\ell$ be the orthogonal projection $\Pi_\ell: L^2(\T) \rightarrow \mathscr{P}_\ell$. 
By the following commutative diagram 
$$
\begin{CD}
L^2(\T) @> \quad M_f \quad >> L^2(\T)  \\
@V{\mathscr{F}}VV @VV {\mathscr{F}}V \\
\ell^2(\Z) @> K_f>> \ell^2(\Z),
\end{CD}
$$
where $M_f$ denotes the operator of multiplication by $f$ and $\mathscr{F}$ denotes the Fourier transform, we obtain 
$$
K_f  = \mathscr{F} \circ M_f  \circ \mathscr{F}^{-1}.
$$
Therefore, by using the identity $\mathscr{F}^{-1} \mathds{1}_{\{0, \cdots, \ell-1\}} \mathscr{F} = \Pi_\ell$, we get
\begin{align*}
 \mathds{1}_{\{0, \cdots, \ell-1 \}} K_f \mathds{1}_{\{0, \cdots, \ell-1\}}  & =  \mathds{1}_{\{0, \cdots, \ell-1 \}}   \mathscr{F} \circ M_f  \circ \mathscr{F}^{-1}  \mathds{1}_{\{0, \cdots, \ell-1\}} 
 \\
 & = \mathscr{F} \circ \Pi_\ell \circ M_f \circ \Pi_\ell \circ \mathscr{F}^{-1}.
\end{align*}
We now argue by contradiction. Suppose that  $\mathds{1}_{\{0, \cdots, \ell-1 \}} K_f \mathds{1}_{\{0, \cdots, \ell-1\}}$ has operator norm equal to $1$.  Then the operator $\Pi_\ell \circ M_f \circ \Pi_\ell $ has operator norm equals to $1$. Since  $\Pi_\ell \circ M_f \circ \Pi_\ell$ acts on the finite dimensional space $\mathscr{P}_\ell$ and is a positive operator, there exists a $P\in \mathscr{P}_\ell\setminus \{0\}$ such that $\Pi_\ell \circ M_f \circ \Pi_\ell  (P) = P$. That is, $P = \Pi_\ell (f P)$. Thus by recalling that $0 \le f \le 1$, we have 
\begin{align*}
\int_\T |P|^2 d\m = \int_\T | \Pi_\ell(f P) |^2 d\m \le  \int_\T | f P |^2 d\m \le \int_\T | P|^2 d\m. 
\end{align*}
Therefore, 
$$
\int_\T (1 -f^2)  | P|^2 d\m = 0.
$$
Consequently  $(1 - f^2) |P|^2 =0$ a.e., then 
$f = 1$ a.e., which contradicts the assumption \eqref{def-sigma}.
The last assertion of the lemma follows immediately. 
\end{proof}

\begin{proof}[Proof of Theorem \ref{prop-syndetic}]
 We will follow \cite{Fan-Schneider} (the second part of the proof of Proposition 6.3 in \cite{Fan-Schneider}). 

Fix $\ell>1$ and $0\le r <\ell$. Consider the random variables
   $$
         Z_{n, \ell} = \mathds{1}_{\{ \X \cap B_{n, \ell}=\emptyset\}}, \quad \forall n \ge 0
   $$
   where $B_{n, \ell} =  [r +n\ell, r +n\ell +\ell -1] \cap \N$, the interval in $\mathbb{N}$ containing $\ell$ consecutive integers with $r +n\ell$ as the starting integer. The variable $Z_{n, \ell}$ describes a gap event, i.e. no point in the interval $B_{n, \ell}$. 
    First notice that
   $$
        \mathbb{E} Z_{n, \ell} = \mathbb{P}(\X \cap B_{n, \ell} =\emptyset) = \det (1 -  \mathds{1}_{B_{n, \ell}} K_f \mathds{1}_{B_{n, \ell}} ) . 
   $$
   This gap probability is independent of  $n$ because of the stationarity of the point process $\X$. Therefore, by Lemma \ref{lem-contractive}, 
    $$
        \mathbb{E} Z_{n, \ell}  = \det (1 -  \mathds{1}_{\{0, \cdots, \ell-1 \}} K_f \mathds{1}_{\{0, \cdots, \ell-1\}} ) > 0.
   $$
   
   Now, we consider  the shift dynamics $(\{0, 1\}^\Z, \nu_f, S)$ associated to $\X$. Note that the dynamical system $(\{0, 1\}^\Z, \nu_f, S^\ell)$ is strongly mixing then totally ergodic, since $(\{0, 1\}^\Z, \nu_f, S)$ is strongly mixing. By the ergodic theorem,  we have the following law of large numbers
   $$
        a.s.  \quad  \lim_{N\to \infty} \frac{1}{N} \sum_{n=1}^N Z_{n, \ell} =\lim_{N\to \infty} \frac{1}{N} \sum_{n=1}^N \phi(S^{\ell n} \xi) = \E Z_{1, \ell}> 0,
   $$
   where $\phi$ is the indicator function of  the cylinder $[0, \cdots,0]_r^{r+\ell -1}$ defined by  $\{\xi\in \{0,1\}^{\Z}: \xi_j=0   \text{ for } r\leq j\leq r+\ell-1 \}$.
   Since $Z_{n, \ell}$'s take value in $\{0, 1\}$, we must have $$\limsup_n Z_{n, \ell} =1 \ \ \  a.s..$$  Then, almost surely,  for any $\ell$ there exists infinitely many $n$'s such that 
   $$
   \X \cap B_{n, \ell} = \emptyset.
   $$
   Since $\ell$ may be arbitrarily large, we obtain that  almost surely, $\X$ is not syndetic. 
      \end{proof}
  
  \subsection{Some remarks} Among the DPPs we are considering, there is the
  $p$-Bernoulli random sequence, as we have already mentioned in Introduction. This $p$-Bernoulli sequence can be generalized in the following way. Consider a sequence of positive numbers $(p_n) \subset (0, 1)$ such that $\sum p_n =\infty$ and a sequence of independent random variables $(\xi_n)$ such that 
  $$P(\xi_n = 1) =p_n = 1- P(\xi_n =0).$$  
  Then consider the infinite random subset of integers
  $$
      \Lambda(\omega) = \{n \ge 1: \xi_n =1\}.
  $$
  
  For this random sequence, the following almost sure facts are known:
  \begin{itemize} 
  	     \item $\# (\Lambda(\omega)\cap [1, N]) \sim \sum_{n=1}^N p_n$ as $N \to \infty$ (see \cite[Proposition 6.2]{Fan-Schneider}). So, $\Lambda(\omega)$
  	     has a positive density $\sigma >0$ iff $ \sum_{n=1}^N p_n \sim \sigma N$ as $N\to \infty$. But in many cases, $\Lambda(\omega)$
  	     has zero density.
  	     \item Kahane and Katznelson \cite[Th\'eor\`eme 1]{KK2008a}  proved that if $n p_n =O(1)$, $\Lambda(\omega)$ is a Sidon set and is discrete and non-dense in the Bohr group, and the closure of 
  	     $\Lambda(\omega)$ in the Bohr group has zero Bohr-Haar measure.  In particular, almost surely
  	     there exist $t = t(\omega) \in (0, 1)$ such that
  	     $$
  	          \lim_{N\to\infty} \frac{\sum_{n=1}^N \xi_n e^{2\pi i n t}}{\sum_{n=1}^N \xi_n} \not= 0.
  	     $$
  	     Thus $\Lambda(\omega)$ is not $L^2$-exact.
  	     \item But Kahane and Katznelson \cite[Th\'eor\`eme B]{KK2008b} 
  	     proved the following opposite result.
  	     If $\lim n p_n =\infty$ and $n p_n$ is slowly varying
  	     (namely, for any $\delta >0$, $n^{1+\delta}p_n$ is increasing and
  	     $n^{1-\delta}p_n$ is decreasing for large $n$), then $\Lambda(\omega)$ is  $L^2$-exact. Fan and Schneider \cite[Theorem 6.1]{Fan-Schneider}   proved the same conclusion under the following simpler conditions
  	     $$
  	     \log N = o(\sum_{n=1}^N p_n), \qquad \sum_{n=1}^N |p_n -p_{n+1}| = o(\sum_{n=1}^N p_n).$$
  	      Notice that the case $p_{2n}=1$ and $p_{2n+1}=0$ is too fluctuant 
  	     and produces the deterministic set  $\Lambda (\omega) = 2 \N$, which is not $L^2$-exact. The above second condition restricts the fluctuation of $(p_n)$. 
  	     \item In \cite[Theorem 6.3]{Fan-Schneider}, it is proved that $\Lambda(\omega)$ is syndedic iff 
  	     $$
  	         \sum_{n=1}^\infty \prod_{j=0}^{\ell-1} (1-p_{n+j}) <\infty.
  	     $$
  	     \item Bourgain \cite[Proposition 8.2]{Bourgain1988} proved that when $p_n = n^{-1}(\log \log n)^B$
  	     with $B > (p-1)^{-1}$, $\Lambda(\omega)$ is good universal for almost
  	     every convergence for functions in $L^p$ ($p>1$).
  	     See Boshernitzan \cite{Boshernitzan1983} for related works.
  	\end{itemize} 
  
  In the following section, we will prove that our $\X_+$ is always
  good unversal for almost everywhere convergence for functions in $L^1$.
  Actually we can deal with polynomial images $P(\X_+)$.

      \section{Pointwise ergodic theorem}
      Let $1\le r\le \infty$. A sequence $\Lambda:=(u_n)_{n\ge 0} \subset \mathbb{N}$ is said to be  {\em good universal  for almost everywhere convergence in $L^r$} if for any measure-preserving dynamical system
      $(X, \mathcal{B}, \nu, T)$ and any $f \in L^{r}(\nu)$, the limit
      $$
           \lim_{N\to \infty}\frac{1}{N}\sum_{k=0}^{N-1} f(T^{u_k} x)
      $$
      exists for $\nu$-almost every $x \in X$. When $(u_k)$ admits a positive density, the above limit in the defintion can be replaced by 
       $$
           \lim_{N\to \infty}\frac{1}{\# (\Lambda \cap [0, N])}\sum_{ u \in \Lambda \cap [0, N]} f(T^{u} x). 
      $$
      \medskip
      
      Set  $\X_+(\omega) = \{u_1(\omega), u_2(\omega), \cdots\}$ with $u_n(\omega)<u_{n+1}(\omega)$ ($\forall n$).  In this section, we will prove 
      that almost surely,  for any   integral polynomial $P \in \Z[x]$
      with $\deg P\ge 1$ verifying  $P(\N) \subset \N$, the sequence $n \mapsto P(u_n(\omega))$ is a good universal  for almost everywhere convergence in $L^r$ for any $r>1$.
      It is also true for $r=1$ when $\deg P =1$.   The following theorem tells a little bit more, because we can distinguish the limit, see (\ref{good-lim}). Recall that, in the following, $\xi_n$ is the random variable
      which takes the value $1$ or $0$ according to $n \in \X_+$ or $n\not\in \X_+$. Recall that $\sigma= \mathbb{E} \xi_n  = \int f(x)dx$ where $f: \mathbb{T}\to [0,1]$ is the function 
      defining the determinantal point process $\X$. 
      
      \begin{theorem}\label{thm-erg}
         There exists an event $\Omega_0$
      with $\mathbb{P}(\Omega_0)=1$ such that each $\omega\in \Omega_0$ has the following property:  for any   polynomial $P \in \Z[X]$ with $\deg P\ge 1$ verifying  $P(\N) \subset \N$, any measure-preserving  dynamical system $(X, \mathcal{B}, \nu, T)$ and any  $g \in L^r(X, \mathcal{B}, \nu)$ ($r>1$),
       the limit
      \begin{equation}\label{good-Lr}
          \lim_{N \to \infty}  \frac{1}{ N} \sum_{n = 0}^{N-1}   (\xi_n (\omega) -  \sigma) g(T^{P(n)} x) =0
      \end{equation}
      holds for $\nu$-almost every $x\in X$ and in  $L^r$-norm. 
      \end{theorem}
      
      We can rewrite (\ref{good-Lr}) as follows
      \begin{equation}\label{good-lim}
           \lim_{N \to \infty}  \frac{1}{ N} \sum_{n = 0}^{N-1}   \xi_n (\omega) g(T^{P(n)} x) =  \sigma   \lim_{N \to \infty}  \frac{1}{ N} \sum_{n = 0}^{N-1}  g(T^{P(n)} x). 
       \end{equation}
      These two limits exist and the existence of the limit on the RHS is ensured by Bourgain \cite[Theorem 1 and Theorem 2]{Bourgain1989-IHES}.

Since polynomials in $\mathbb{Z}[x]$ are countable, it suffices to find an event $\Omega_0$  for any fixed polynomial $P$. With a fixed $P$, 
Theorem \ref{thm-erg} follows immediately from the following result of Fan
\cite[Corollary 2]{Fan-Osc} applied to $w_n = \xi_n -\sigma$ and Theorem \ref{prop-quasi-G} below. 

\begin{proposition} \cite[Corollary 2]{Fan-Osc} \label{Fan}
Let  $P \in \Z[x]$ be a polynomial with integer coefficient such that $P(\N) \subset \N$ and let  $(w_n)_{n= 0}^\infty$ is a bounded sequence in $\C$.
Suppose that there exist $C > 0, \varepsilon > 0$ such that for any $N \in \N$, we have
\begin{align}\label{D-cond}
\max_{0 \le t \le 1} \left|  \sum_{n=0}^{N-1} w_n e^{2 \pi i P(n) t}  \right| \le C \frac{N}{(\log N)^{1/2 + \varepsilon}}.
\end{align}
Then for any measure-preserving dynamical system  $(X, \mathcal{B}, \nu, T)$,  for any $r > 1$ and any $g \in L^r(X, \mathcal{B}, \nu)$, the limit
      $$
      \lim_{N \to \infty}  \frac{1}{ N} \sum_{n = 0}^{N-1}    w_n g(T^{P(n)} x)   = 0   
      $$ 
   holds for $\nu$-almost every $x\in X$ and in  $L^r$-norm. 
\end{proposition}

\begin{theorem}\label{prop-quasi-G} Let    $P \in \Z[x]$ be a polynomial  of degree $d$ such that $P(\N) \subset \N$.
	For $N$ large enough, we have 
	\begin{align}\label{inq-LD}
	\PP\left(  \max_{0 \le t \le 1} \left|  \sum_{n=0}^{N-1} (\xi_n(\omega) - \sigma) e^{2 \pi i P(n) t}\right|   \ge 100 \sqrt{ d N \log N} \right) \le \frac{1}{N^2}.
	\end{align}
Therefore, almost surely, there exists $C_\omega > 0$, such that for any $N \ge 2$, 
\begin{align}\label{D-cond-DPP}
\max_{0 \le t \le 1} \left|  \sum_{n=0}^{N-1} (\xi_n(\omega) - \sigma) e^{2 \pi i P(n) t}  \right| \le C_\omega   \sqrt{d N   \log N}.
\end{align}
\end{theorem}

The proof of Theorem \ref{prop-quasi-G} will be based on the following  sub-Gaussian property of the random variables 
\[
\sum_{n = 0}^{N-1}   (\xi_n(\omega) - \sigma) a_n.
\]
where $a_0, a_1, \cdots, a_{N-1} \in \R$. 

\begin{proposition}\label{lem-qG}
For any $N\in \N$ and $\lambda \in \R$  and  $ a: = (a_0, \cdots, a_{N-1}) \in \R^N$, we have
\begin{align}\label{inq-qG}
\E \left[ \exp\left(\lambda \sum_{n = 0}^{N-1}   (\xi_n(\omega) - \sigma) a_n \right ) \right] \le \exp \left(  \lambda^2 \sum_{n=0}^{N-1} a_n^2\right).
\end{align}
\end{proposition}

The proof of Proposition \ref{lem-qG} is a consequence of the negative association property proved by Lyons \cite[Theorem 6.5]{DPP-L} of the determinantal point processes. Here is  the definition of the negative association.  A function $f: \{0,1\}^{\Z}\to \R$ is said to be {\it increasing} if for  two  points $\xi, \eta \in\{0,1\}^{\Z}$ such that $\xi_n\leq \eta_n$ for all 
 $n\in \Z$, we have $f(\xi)\leq  f(\eta)$, and it is said to be {\it decreasing} if $-f$ is increasing.  A probability on $\{0,1\}^{\Z}$ is said to have {\it negative associations} if for any pair $f_1, f_2$ of increasing functions that are measurable with respect to complementary subsets of $\Z$ (i.e. $f_1$ and $f_2$ depend respectively on $\{\xi_i:i\in A\}$ and  $\{\xi_j: j\in B\}$ with $A\cap B=\emptyset$), we have ${\rm Cov} (f_1, f_2)\le 0$, namely
$$\E[ f_1f_2] \leq \E[ f_1]\E[ f_2] .$$
Note that the product of increasing {\bf nonnegative} functions is still  an increasing nonnegative function.
 When the probability has negative associations, for any collection $f_1 , f_2 , ..., f_n$ of increasing  nonnegative functions that are
measurable with respect to pairwise disjoint subsets of $\Z$, we have 
 $$ \E[f_1 f_2 \cdots f_n] \leq \E[f_1]\E[f_2]\cdots\E[f_n].$$
Remark that the negative association property also implies that the above inequality holds for decreasing nonnegative functions as well.



\begin{proof}[Proof of Proposition \ref{lem-qG}]
	We will need the following elementary inequality:
	for any $u \in [-1, 1]$ and  $x\in \R$, we have 
	\begin{align}\label{inq-convexity}
	e^{ux} \le \frac{1+u}{2} e^x + \frac{1-u}{2} e^{-x}.
	\end{align}
	It is a direct consequence of the convexity of the exponential function and 
	the fact $u x = \frac{1 + u}{2} x + \frac{1-u}{2} (-x)$. 
	
Fix $N\in \N$, $\lambda \in \R$  and  $ a: = (a_0, \cdots, a_{N-1}) \in \R^N$. We divided  $\{0, 1, \cdots, N-1\}$ into two subsets following the sign of $\lambda a_n$ as follows:
$$
I^{+} : = \{0 \le n \le N-1:  \lambda a_n \ge 0 \}, \quad I^{-}: = \{0 \le n \le N-1:\lambda a_n < 0\}.
$$
Then we may write 
\[
 \exp \left( \sum_{n = 0}^{N-1}   (\xi_n(\omega) - \sigma) \lambda a_n\right)  =  \prod_{n \in I^{+}}  e^{ (\xi_n(\omega) - \sigma) \lambda a_n}   \prod_{n \in I^{-}}  e^{ (\xi_n(\omega) - \sigma) \lambda a_n}.
\]
By Cauchy-Schwarz inequality, we obtain 
\[
\E \left[ \exp\left(\lambda \sum_{n = 0}^{N-1}   (\xi_n(\omega) - \sigma) a_n \right ) \right]  \le  \E \left(  \prod_{n \in I^{+}}  e^{ 2 (\xi_n(\omega) - \sigma) \lambda a_n}   \right)^{1/2}  \E \left(  \prod_{n \in I^{-}}  e^{ 2 (\xi_n(\omega) - \sigma) \lambda a_n}   \right)^{1/2}.
\]
Note that all the functions $\xi \mapsto e^{ 2 (\xi_n(\omega) - \sigma) \lambda a_n}$ with $n \in I^{+}$ are {\it positive}  and increasing, while all the functions $\xi \mapsto e^{ 2 (\xi_n(\omega) - \sigma) \lambda a_n}$ with $n \in I^{-}$ are {\it positive} and decreasing,  hence by the negative association of the determinantal process, we have 
\[
 \E \left(  \prod_{n \in I^{+}}  e^{ 2 (\xi_n(\omega) - \sigma) \lambda a_n}   \right)  \le  \prod_{n \in I^{+}}  \E \left(   e^{ 2 (\xi_n(\omega) - \sigma) \lambda a_n}   \right)
 \]
 and
 \[
  \E \left(  \prod_{n \in I^{-}}  e^{ 2 (\xi_n(\omega) - \sigma) \lambda a_n}   \right)  \le  \prod_{n \in I^{-}}  \E \left(   e^{ 2 (\xi_n(\omega) - \sigma) \lambda a_n}   \right).
\]
Therefore, 
\begin{align}\label{na-inq}
\E \left[ \exp\left(\lambda \sum_{n = 0}^{N-1}   (\xi_n(\omega) - \sigma) a_n \right ) \right]  \le \left (  \prod_{n =0}^{N-1}  \E \left(   e^{ 2 (\xi_n(\omega) - \sigma) \lambda a_n}   \right)\right)^{1/2}.
\end{align}
Now use the fact that $\xi_n - \sigma  \in [-1, 1]$ and the inequality \eqref{inq-convexity}, we have 
\[
e^{ 2 (\xi_n(\omega) - \sigma) \lambda a_n}  \le \frac{1 + (\xi_n - \sigma)}{2} e^{2 \lambda a_n }   + \frac{1 - (\xi_n - \sigma)}{2} e^{- 2 \lambda a_n}
\]
and hence 
\[
\E \left( e^{ 2 (\xi_n(\omega) - \sigma) \lambda a_n} \right) \le \frac{e^{2 \lambda a_n }   + e^{- 2 \lambda a_n}}{2}.
\]
Using the following elementary inequality 
\[
\frac{e^x  + e^{-x}}{2} = \sum_{k = 0}^\infty \frac{x^{2k}}{(2k)!} \le \sum_{k = 0}^\infty \frac{x^{2k}}{2^k k!}  = e^{x^2/2}, 
\]
we obtain 
\begin{align}\label{inq-single}
\E \left( e^{ 2 (\xi_n(\omega) - \sigma) \lambda a_n} \right) \le e^{2 \lambda^2 a_n^2}.
\end{align}
Combining the inequalities \eqref{na-inq} and \eqref{inq-single}, we obtain the desired inequality \eqref{inq-qG}.
\end{proof}


	\begin{proof}[Proof of Theorem \ref{prop-quasi-G}]
	Our proof is the combination of an idea of Salem-Littlewood developped by Kahane \cite{Kahane-random} and the above sub-gaussian property (\ref{inq-qG}). 
Set 
$$
Q_\omega(t)  = \sum_{n=0}^{N-1} (\xi_n(\omega) - \sigma) \cos(2\pi  P(n) t), \quad t \in [0, 1).
$$
By Bernstein's inequality for trigonometric polynomials, we have 
$$
\|Q_\omega'\|_\infty \le  (N-1)^d \| Q_\omega\|_\infty \le N^d \| Q_\omega\|_\infty.
$$
Now let  $t_{max}(\omega) \in [0, 1)$ be such that $| Q_\omega(t_{max})| = \| Q_\omega \|_\infty$ (Here, we do not need the information whether $\omega \mapsto t_{max}(\omega)$ is measurable or not). Set 
$$
I(\omega) : = \{t \in [0, 1):  | t -t_{max}(\omega)| \le \frac{1}{2N^d} \}.
$$ 
Then for any $t \in I(\omega)$, we have 
$$
\left| Q_\omega(t) - Q_\omega(t_{max}(\omega) ) \right| \le \|Q_\omega'\|_\infty\cdot | t - t_{max}(\omega)| \le \frac{\| Q_\omega \|_\infty}{ 2},
$$
which implies that for any $t \in I(\omega)$, $| Q_\omega(t)| \ge \frac{\| Q_\omega\|_\infty}{2}$ and hence 
$$
e^{\lambda  \frac{\|Q_\omega \|_\infty}{2}} \le e^{\lambda  Q_\omega(t)}  + e^{-\lambda Q_\omega(t)}.
$$
It follows that 
$$
e^{\lambda  \frac{\|Q_\omega \|_\infty}{2}}  \le \frac{1}{| I(\omega)|} \int_{I(\omega)} \left( e^{\lambda  Q_\omega(t)}  + e^{-\lambda Q_\omega(t)}\right) dt \le 2N^d \int_{[0, 1)} \left( e^{\lambda  Q_\omega(t)}  + e^{-\lambda Q_\omega(t)}\right) dt.
$$
Therefore, by using Proposition \ref{lem-qG} (using the inequalities $| \cos (2 \pi P(n) t) | \le 1$), we obtain that for any $\lambda \in [-1/2, 1/2]$, 
\begin{align*}
\E\left(e^{\lambda  \frac{\|Q_\omega \|_\infty}{2}}  \right) \le 2 N^d  \int_{[0, 1)} \left( \E e^{\lambda  Q_\omega(t)}  + \E e^{-\lambda Q_\omega(t)}\right) dt \le 4N^d  \exp( \lambda^2  N ).
\end{align*}
By Markov-Chebychev's inequality, for any $\Delta > 0$, we have 
\begin{align*}
\PP\left( \| Q_\omega\|_\infty \ge \Delta \right)  \le \frac{\E\left(e^{\lambda  \frac{\|Q_\omega \|_\infty}{2}}  \right) }{e^{\lambda \Delta/2}} \le   4N^d  \exp\left( \lambda^2  N - \lambda \frac{\Delta}{2}\right).
\end{align*}
Taking $\lambda = \frac{\Delta}{4N}$, we obtain 
\begin{align*}
\PP\left( \| Q_\omega\|_\infty \ge \Delta \right) \le   4N^d  \exp\left(- \frac{\Delta^2}{16N}\right).
\end{align*}
Now by taking $\Delta>0$ such that 
$$
4N^d  \exp\left(- \frac{\Delta^2}{16N}\right) = \frac{1}{2 N^2}, 
$$
we obtain $\Delta = 4 \sqrt{ N \log (8N^{d+2})} $ and hence 
\begin{align*}
\PP\left( \| Q_\omega\|_\infty \ge 4 \sqrt{ N \log (8N^{d+2})}  \right) \le \frac{1}{ 2N^2}, \quad \text{for $N$ large.}
\end{align*}
By similar arguments, let
$$
P_\omega(t)  = \sum_{n=0}^{N-1} (\xi_n(\omega) - \sigma) \sin(2\pi  P(n) t), \quad t \in [0, 1),
$$
then
\begin{align*}
\PP\left( \| Q_\omega\|_\infty \ge 4\sqrt{N \log (8N^{d+2})}  \right) \le \frac{1}{ 2N^2}, \quad \text{for $N$ large.}
\end{align*} 
Consequently, for large enough $N$, we have
\begin{align*}
& \PP\left( \| Q_\omega + i P_\omega \|_\infty \ge 8 \sqrt{ N \log (8N^{d+2})}    \right) 
\\
\le &\PP\left( \| Q_\omega\|_\infty \ge 4 \sqrt{ N \log (8N^{d+2})}  \right) + \PP\left( \| P_\omega\|_\infty \ge 4 \sqrt{ N \log (8N^{d+2})}  \right)
\\
\le & \frac{1}{N^2}.
\end{align*}
Since for large $N$, we have $100 \sqrt{d N \log N } \ge 8 \sqrt{N \log (8N^{d+2})}$, we obtain the desired inequality \eqref{inq-LD}.

By Borel-Cantelli lemma, the inequality \eqref{inq-LD} implies that for $\PP$-a.e. $\omega$, when $N$ is large enough, we have 
$$
\max_{0 \le t \le 1} \left|  \sum_{n=0}^{N-1} (\xi_n(\omega) - \sigma) e^{2 \pi i P(n) t}\right|   \le 100 \sqrt{ d N \log N}.
$$ 
This implies the desired domination \eqref{D-cond-DPP}. 
\end{proof}

\section{Khintchine-Kahane inequalitiy}
In this section, we will obtain a Khintchine-Kahane inequality and will apply it to study  the almost everwhere convergence of   the random series:
\begin{align}\label{random-sum}
\sum_{n} a_n  (\xi_n - \sigma), 
\end{align}
where $a  = (a_n)_{n\in\Z}\in \ell^2(\Z)$ is a square summable sequence in $\C$.

We first show that the series \eqref{random-sum} defines a random variable for all $a \in \ell^2(\Z)$. More precisely, we have the following result.

\begin{theorem}\label{thm-L2}
Assume $f: \T \rightarrow [0, 1]$ with $\sigma := \int f d\m \in (0, 1)$.
For any complex sequence $a= (a_n)_{n\in\Z}\in \ell^2(\Z)$ we have 
\begin{align}\label{apriori-maj}
\left\|\sum_n a_n (\xi_n - \sigma) \right\|_2 \le   \sqrt{\sigma} \left( \sum_n | a_n|^2 \right)^{1/2}
\end{align}
and then  the series $\sum_n a_n (\xi_n - \sigma)$ converges in $L^2$-mean.
If $f$ is not an indicator function, we have the inverse inequality
\begin{align}\label{L2-min}
\left\|\sum_n a_n (\xi_n - \sigma) \right\|_2 \ge c_f \left( \sum_n | a_n|^2 \right)^{1/2}
\end{align}
for all $a \in \ell^2(\Z)$, where 
\begin{align}\label{c-f}
c_f = \sqrt{\int_{\T} f (1-f) d\m}> 0. 
\end{align}
(We have $c_f>0$ if and only if $f$ is not an indicator function).
\end{theorem}

\begin{proof}
Note first that $\xi_n^2 = \xi_n$ for any $n\in \Z$. Therefore 
\begin{equation}\label{Riesz-1}
\E((\xi_n-\sigma)^2) = \E( \xi_n - 2 \sigma \xi_n + \sigma^2) = \sigma - \sigma^2 = \sigma - | \widehat{f}(0)|^2.
\end{equation}
If $n \ne m$, we have already seen that
\begin{equation}\label{Riesz-2}
\E(\xi_n \xi_m)=   \det \left(\begin{array}{cc}  \widehat{f}(n-n)& \widehat{f}(n-m) \\ \widehat{f}(m-n) &\widehat{f}(m-m)  \end{array}\right) = \sigma^2 - | \widehat{f}(n-m)|^2.
\end{equation}
Therefore, 
$$
\E((\xi_n-\sigma) (\xi_m-\sigma))= - | \widehat{f}(n-m)|^2.
$$

To prove (\ref{apriori-maj}), we have only to prove it for finitely supported sequences $a$. By (\ref{Riesz-1}) and (\ref{Riesz-2}), we have 
\begin{align}\label{2-norm-f}
\begin{split}
\left\|\sum_n a_n (\xi_n - \sigma) \right\|_2^2   & = \sum_{n, m} a_n \overline{a}_m \E((\xi_n-\sigma) (\xi_m-\sigma)) 
\\
& = \sigma \sum_{n} | a_n|^2 - \sum_{n, m} a_n \overline{a}_m  |\widehat{f}(n-m)|^2.
\end{split}
\end{align}
Define $f^{\vee}(t) = f(-t)$ for $t \in \R/\Z$. Then 
$
|\widehat{f}|^2 = \widehat{f* f^\vee},
$
where $f * f^\vee$ denotes the convolution of $f$ and $f^\vee$, defined by
$$
f * f^\vee (s) 
= \int_{\R/\Z} f(s+t) f(t) dt.  
$$
We note that $f*f^\vee$ is a continuous function, and it  takes values in $[0, \sigma]$ because $0\le  f \le 1$ and $\int f d\m =\sigma$.
It follows that, the operator $K_{f*f^\vee}: \ell^2(\Z) \rightarrow \ell^2(\Z)$ appearing in the following commuting diagram 
$$
\begin{CD}
L^2(\T) @> \quad M_{f*f^\vee} \quad >> L^2(\T)  \\
@V{\mathscr{F}}VV @VV {\mathscr{F}}V \\
\ell^2(\Z) @> K_{f*f^\vee}>> \ell^2(\Z),
\end{CD}
$$
has an infinite matrix representation 
$$
\Big[| \widehat{f}(m-n)|^2\Big]_{m, n\in \Z}
$$
and its operator norm is equal to $\|f * f^\vee\|_\infty$ which is the norm of $M_{f*f^\vee}$: 
$$
\| K_{f*f^\vee}\| = \| f * f^\vee\|_\infty \le \sigma.
$$

The equality \eqref{2-norm-f} can be reformulated as 
\begin{equation}\label{Riesz-3}
\left\|\sum_n a_n (\xi_n - \sigma) \right\|_2^2   = \left\langle \Big(\sigma {\rm Id} - K_{f * f^\vee}\Big) a, a \right\rangle_{\ell^2(\Z)}.
\end{equation}
The positivity of the operators $K_{f*f^\vee}$ implies immediately 
\[
\left\langle \Big(\sigma {\rm  Id} - K_{f * f^\vee}\Big) a, a \right\rangle_{\ell^2(\Z)}\le  \sigma \| a \|_2^2,
\]
which implies the desired inequality \eqref{apriori-maj}.

For proving the inverse inequality \eqref{L2-min}, note that 
\begin{align*}
\| f*f^\vee\|_\infty = \int_{\R/\Z} f(t)^2 dt.
\end{align*}
Indeed, this follows from the continuity of  $f*f^\vee$, the fact $f*f^\vee (0) = \int_{\R/\Z} f(t)^2 dt$ and the estimate
\[
\forall s \in \R/\Z, \ \ \
f*f^\vee (s) 
\le   \left(\int_{\R/\Z} f(s + t)^2 dt\right)^{1/2}  \left( \int_{\R/\Z}  f(t)^2 dt \right)^{1/2} = \int_{\R/\Z} f(t)^2 dt.
\]
It follows that (with $c_f$ defined in \eqref{c-f})
\[
\sigma  {\rm Id} - M_{f * f^\vee} \ge  c_f^2 \cdot  {\rm Id}
\ \ \ {\rm or \ equivalently} \ \ \ 
\sigma  {\rm Id} - K_{f * f^\vee} \ge  c_f^2 \cdot  {\rm Id},
\]
which, together with (\ref{Riesz-3}), implies the desired inequality \eqref{L2-min}.

Now we prove that   the inequality \eqref{L2-min} holding for some constant $c>0$ implies $f$ is not an indicator function.
Assume that the inequality \eqref{L2-min} holds for some constant $c>0$. Then 
$$
\left\langle \Big(\sigma  {\rm Id}  - K_{f * f^\vee}\Big) a, a \right\rangle_{\ell^2(\Z)}\ge  c\| a \|_2^2 =c \left\langle a, a \right\rangle_{\ell^2(\Z)}, 
\quad \forall a \in \ell^2 (\Z).
$$
This implies $\sigma  {\rm Id} - K_{f * f^\vee} \ge c  {\rm Id}$ and hence 
\[
\int_{\R/\Z} f(t)^2 dt = \| K_{f * f^\vee}\| \le \sigma -c  < \int_{\R/\Z} f(t)dt.
\]
The above strict inequality holds (under the assumption $0\le f \le 1$) if and only if $f$ is not an indicator function. 
\end{proof}

Now let us prove the Khintchine inequality for the random sequence 
$(\xi_n -\sigma)$. 

\begin{theorem}\label{thm-KK}
Assume that $f$ is not an indicator function. Then for any $p \in [2, \infty)$, there exists $C = C_f(p)> 0$ such that
\begin{align}\label{KK-inq}
\left\| \sum_{n} a_n  (\xi_n - \sigma) \right\|_2 \le \left\| \sum_{n} a_n  (\xi_n - \sigma) \right\|_p \le C_f(p)\left\| \sum_{n} a_n  (\xi_n - \sigma) \right\|_2, \quad \forall a \in \ell^2(\Z).
\end{align} 
 Here we may take 
$$
C_f(p) = \frac{2\sqrt{2} \Gamma\left(\frac{p}{2}+1\right)^{\frac{1}{p}}}{ \sqrt{\int_{\R/\Z} f(1-f) d\m}}.
$$
\end{theorem}

Theorem \ref{thm-KK} follows immediately from Theorem \ref{thm-L2} and Lemma \ref{lem-p-norm} below. We emphasize that the upper estimate  of the $L_p$-norm given in Lemma \ref{lem-p-norm} holds for all $f : \T \rightarrow [0, 1]$ with $\sigma \in (0, 1)$ (that is, without the assumption that $f$ is not an indicator function).

\begin{lemma}\label{lem-p-norm}
For any $f: \T \rightarrow [0, 1]$ with $\sigma = \int f  d\m \in (0, 1)$ and any $p \in [1, \infty)$, we have 
\begin{align*}
 \left\| \sum_{n} a_n  (\xi_n - \sigma) \right\|_p \le   2\sqrt{2} \Gamma\left(\frac{p}{2}+1\right)^{\frac{1}{p}} \left(\sum_n a_n^2\right)^{1/2}, \quad \forall a \in \ell^2(\Z).
 \end{align*}
\end{lemma}
\begin{proof}
It suffices to show 
\begin{align}\label{real-KK}
 \left\| \sum_{n} a_n  (\xi_n - \sigma) \right\|_p \le   \Gamma\left(\frac{p}{2}+1\right)^{\frac{1}{p}} \left(\sum_n a_n^2\right)^{1/2}
 \end{align}
 for real sequence $a\in \ell^2(\Z)$.
Now fix a non-zero  real sequnce $a \in \ell^2 (\Z)$.
Proposition \ref{lem-qG} and the Markov inequality imply that for any $t > 0$ and any $\lambda>0$, we have
 \[
 \PP\left(  \left| \sum_{n} a_n  (\xi_n - \sigma) \right|\ge t  \right)  \le \frac{\E \left[ \exp\left(\lambda \sum_{n}    (\xi_n(\omega) - \sigma) a_n \right ) \right]}{e^{\lambda t}}\le  \exp \left( \lambda^2 \sum_{n} a_n^2 -\lambda t\right). 
 \]
 Taking $\lambda = \frac{t}{2 \sum_n a_n^2}$, we obtain 
 \[
 \PP\left(  \left| \sum_{n} a_n  (\xi_n - \sigma) \right|\ge t  \right)  \le \exp\left(  -\frac{t^2}{4\sum_n a_n^2}\right).
 \]
 It follows that 
 \begin{align*}
 \left\| \sum_{n} a_n  (\xi_n - \sigma) \right\|_p^p & = p \int_{0}^\infty t^{p-1} \PP\left(  \left| \sum_{n} a_n  (\xi_n - \sigma) \right|\ge t  \right) dt
 \\
 & \le p \int_{0}^\infty t^{p-1}  \exp\left(  -\frac{t^2}{4\sum_n a_n^2}\right) dt 
 \\
& = 2^p \Gamma(p/2+1) \left(\sum_n a_n^2\right)^{p/2}.
 \end{align*}
  This implies the desired inequality \eqref{real-KK}.
  Such an argument was used in \cite{Fan2017ETDS}.
\end{proof}

\section{An arithmetic property}
Recall that   the sum-set $A+B$ of  two given sets $A, B\subset \Z$ is defined by 
\[
A + B: = \{a + b \in \Z \, |a \in A, b \in B \}.
\]

We have seen that the random set $\X$ is rich to some extent (intersecting every arithmetic sequence),
 but is not so rich to some other extent (not syndetic).
 In this section, we prove that its sum with itself is the whole set of integers.
  
\begin{theorem}\label{arithm}
For $\PP$-a.e. $\omega$, the random subset $\X = \X(\omega)$ satisfies
\[
\X  + \X = \Z.
\]
\end{theorem}
\begin{proof} Since $\Z$ is countable, we have only to show that 
	$\mathbb{P} (m \in \X + \X) =1$ for any $m \in \Z$.
Fix an arbitrary $m \in \Z$. Let
\[
M : = 1+ | m|. 
\]
For any $n \in \N $, consider the event
\[
A_n: = \left\{  \xi_{ n M } = 1, \xi_{m-n M}  =1 \right\} =\{ nM, m- nM \in \X\}.
\]
First observe that by our choice of $M$, we have
\[
n \in \N \an \{ n M , m -n M\} \cap \{n' M, m -n'M\} \ne \emptyset \Longrightarrow n = n'. 
\]
Also observe that $A_n$'s are increasing event. 
Therefore, 
by the negative association of the determinantal point process  $\X$ (see Lyons \cite[Theorem 3.7]{Lyons-ICM}),  the family $\{A_n: n \in \N\}$ are pairwise negatively correlated, that is, 
\begin{equation}\label{XX-1}
\PP(A_n \cap A_{n'}) \le \PP(A_n)\PP(A_{n'}) \ \ \ \forall n, n' \in \N \an n \ne n'.
\end{equation}
We also have 
$$
\PP(A_n) = \E( \xi_{ n M } \xi_{m-n M} ) =  | \widehat{f}(0)|^2 - | \widehat{f}(m -2 n M)|^2
$$
and $ \sum_n \widehat{f}(m -2 n M)|^2 \le \|f\|_2^2<\infty$, hence 
\begin{equation}\label{XX-2}
\sum_{n\in \N} \PP(A_n) = + \infty.
\end{equation}
By (\ref{XX-1}), (\ref{XX-2}) and  a generalized Borel-Cantelli lemma (see, e.g., \cite[Theorem 1]{Yan-BC}), we have
\[
\PP(\limsup A_n) = 1.
\]
By the definition of $A_n$,  we have  
\(
\{m \in \X + \X \} \supset A_n
\) for any $n\in \N$.
Therefore, 
$$
     \limsup A_n \subset \{m \in \X +\X\}
$$
so that $\PP(m \in \X + \X) = 1$. 
\end{proof}

\section{Generalization to DPP on $\Z^d$}
The results that we have obtained for determinantal point processes  on $\mathbb{Z}$
also holds for  determinantal point processes  on $\mathbb{Z}^d$ ($d \ge 1$).

The last result (Theorem \ref{arithm}) is generalized to 
$$
    \X +\X = \mathbb{Z}^d.
$$
The non syndetic property of $\X$ on $\mathbb{Z}^d$ (generalization of Theorem \ref{prop-syndetic}) states that there are arbitrarily large balls
in which there is no points from $\X$.

For DPP on $\Z^d$ we can similarly prove the sub-Gaussian property, the inequality of Salem-Littlewood type and the inequality of Khintchine-Kahane. 
 
 Theorem \ref{thm-erg} can be generalized as follows.
Let $T_1, \cdots, T_d$ be $d$ commutative measure-preserving transformations
on the probability space $(X, \mathcal{B}, \nu)$.
For any $f \in L^1(\nu)$, it is known that the following limit
$$
    \lim_{n\to\infty}\frac{1}{n^d} \sum_{0\le k_1, \cdots, k_d <n} f(T_1^{k_1}\cdots T^{k_d}_d x)
$$
exists $\nu$-a.e.
Let $\X$ be a DPP on $\Z^d$, we would like to investigate the existence
of the weighted erodic limit
\begin{equation}\label{U}
\lim_{n\to\infty}\frac{1}{\# (\X \cap [0, n]^d) } \sum_{(k_1, \cdots, k_d) \in \X \cap [0, n]^d} f(T_1^{k_1}\cdots T^{k_d}_d x).
\end{equation}
We can prove a result similar  to Theorem \ref{thm-erg} for 
the limit (\ref{U}) and for $f\in L^1(\mu)$.
To prove this, we need a generalization of Proposition \ref{Fan}.
Following \cite{Fan-Osc} to prove such a generalization, we only need the following generalized theorem of Davenport-Erd\"{o}s-LeVeque
\cite{DEL}.

\begin{theorem}\label{HighDEL} Let $(\xi_{{\bf k}})_{{\bf k} \in \N^{d}}$ be a collection of  bounded random variables with  $\sup\limits_{{\bf k}\in \N^d}\|\xi_{\bf k}\|_\infty <+\infty$.
	Let 
	$$
	X_{n}=n^{-d}\sum\limits_{k_1,\cdots, k_d \leq n}\xi_{k_1,\cdots, k_d}.
	$$
	We have   $\lim_{n\to \infty}X_n=0$ almost surely if 
  $$\sum_{n=1}^{\infty}\frac{\E|X_{n}|^2}{n}<+\infty.
  $$   

\end{theorem}

\begin{proof} Let us first recall the following elementay fact.
 If $(a_n)_{n\geq 1}$ is a sequence of positive real numbers
 such that   $\sum_{n\geq 1}a_n<\infty$, then there exists  an increasing positive sequences of real number $\lambda_n$ tending to infinity such that
$\sum_{n\geq 1}\lambda_n a_n< +\infty.$  
Actually, we can take $\lambda_n=\frac{1}{\sqrt{r_n}+\sqrt{r_{n+1}}}$ 
where  $r_n=\sum_{k\geq r} a_k$.
Thus, by the hypothesis, there exists an increasing sequence
of positive number $(\lambda_n)$ tending to infinity such that 
$$\sum_{n\geq 1} \frac{\E|X_n|^2}{n}\lambda_n<+\infty.$$
We can assume that $\lambda_n>1$. 
Define recursively
$$  M_1=1, \quad \forall r\ge 1, M_{r+1}=\left[\frac{\lambda_{M_r}}{\lambda_{M_r}-1}Mr\right]+1.$$
We have $M_1<M_2<\cdots <M_r< \cdots$, since $\frac{\lambda_{M_r}}{\lambda_{M_r}-1}>1$.
It is also easy to see that 
\begin{align}\label{Lambda}
\frac{M_{r+1}}{M_{r+1}-M_{r}}\leq \lambda_{M_{r}}.
\end{align}
Let $M_r<n_r\leq M_{r+1}$ be an integer such that 
$$\E|X_{n_r}|^2=\min_{M_r<n\leq M_{r+1}}\E |X_n|^2.$$
Then 
\begin{align*}
\E|X_{n_r}|^2&\leq \frac{1}{M_{r+1}-M_r}\sum_{n=M_r+1}^{M_{r+1}}\E|X_n|^2 
\leq  \frac{M_{r+1}}{M_{r+1}-M_r} \sum_{n=M_r+1}^{M_{r+1}}\frac{\E|X_n|^2}{n} 
\leq  \sum_{n=M_r+1}^{M_{r+1}}\frac{\E|X_n|^2}{n}\lambda_n   
\end{align*}
where we have used (\ref{Lambda}) for the last inequality. 
So $$\sum_{r\geq 1} \E|X_{n_r}|^2 \leq \sum_{n=1}^{\infty}\frac{\E|X_n|^2}{n}\lambda_n<+\infty.$$
It follows that   $\lim_{n\to \infty}X_{n_r}=0$ almost surely. 

Notice that $M_{r+1}\sim M_r$ for  $M_{r+1}\sim \frac{\lambda_{m_r}}{\lambda_{m_r}-1}M_r.$ So 
\begin{align}\label{sim1}
\lim_{r\to \infty}\frac{n_{r+1}}{n_r}=1.
\end{align}
Now we interpolate.  For any $n_r<n\leq n_{r+1}$, 
we have 
$$\left|\sum_{k_1,\cdots, k_d\leq n}\xi_{k_1,\cdots,k_d}-\sum_{k_1,\cdots, k_d\leq n_r}\xi_{k_1,\cdots,k_d}\right|\leq \sup_{k} \|\xi_{k_1,\cdots,k_d}\|_{\infty}(n_{r+1}^d-n_r^d).$$
So, by (\ref{sim1}), we have
$$\left|\frac{1}{n^d}\sum_{k_1,\cdots, k_d\leq n}\xi_{k_1,\cdots,k_d}-\frac{n_r^d}{n^d}\frac{1}{n_r^d}\sum_{k_1,\cdots, k_d\leq n_r}\xi_{k_1,\cdots,k_d}\right|\leq \sup_{k} \|\xi_{k_1,\cdots,k_d}\|_{\infty}\frac{n_{r+1}^d-n_r^d}{n^d}.$$
Then, almost surely, 
$$\lim_{n\to \infty}\frac{1}{n^d}\sum_{k_1,\cdots, k_d\leq n}\xi_{k_1,\cdots,k_d}=\lim_{n_r\to \infty}\frac{1}{n_r^d}\sum_{k_1,\cdots, k_d\leq n_r}\xi_{k_1,\cdots,k_d}=0.$$
\end{proof}

%

\medskip
{\em Acknowledgement.}  The first author would like to thank  Institut Mittag-Leffler where  the work was finished during 
2017 fall term research program Fractal Geometry and Dynamics.

\end{document}